\documentclass[12pt]{article}
\usepackage[utf8]{inputenc}
\usepackage[T1]{fontenc}
\usepackage[french]{babel}
\usepackage{fullpage}
\usepackage{amsmath}
\usepackage{url}
\usepackage{titling}
\usepackage{amsthm}
\usepackage{tikz}
\usepackage{tikz-cd}
\usepackage{amssymb}
\usepackage{listings}
\usepackage{longtable}
\usepackage{hyperref}
\hypersetup{
    colorlinks,
    citecolor=black,
    filecolor=black,
    linkcolor=black,
    urlcolor=black
}
\usetikzlibrary{patterns}

\newcommand{\mbb}{\mathbb}
\newcommand{\mc}{\mathcal}

\newcommand{\cpx}{\mbb{C}}

\newcommand{\ring}{\mbb{A}}
\newcommand{\sym}{\mathfrak S}

\newcommand{\nat}{\mbb{N}}
\newcommand{\op}{\operatorname}

\newcommand{\abs}[1]{\left| #1 \right|}
\newcommand{\arw}[1]{\overset{#1}{\longrightarrow}}
\newcommand{\dual}{\widehat}
\newcommand{\Hom}{\operatorname{Hom}}

\newtheorem*{defin}{Définition}

\newtheorem*{remark}{Remarque}

\newtheorem*{prop}{Proposition}
\newtheorem*{corol}{Corollaire}
\newtheorem*{theorem}{Théorème}

\newtheorem*{lem}{Lemme}

\newtheorem*{conjecture}{Conjecture}

\begin{document}
\title{Arithmétique des Groupes Abéliens Finis}

\author{Louis Mallet-Burgues}
\maketitle
\tableofcontents
\section{Introduction}
Le but de l'article est de présenter diverses techniques pour manipuler les groupes abéliens finis. On introduit un analogue de la convolution de Dirichlet qui permet d'obtenir des résultats combinatoires sur les groupes abéliens finis. \\\\ 
\textit{Il se trouve que cet outil avait déjà été introduit par Delsarte dans} \cite{delsarte}, \textit{chose dont je me suis rendu compte en en discutant avec un collègue.}
\\\\
On utilise ensuite la notion d'action régulière pour obtenir un fait surprenant : le nombre de parties génératrices d'un groupe abélien fini $G$ est divisible par le cardinal de $G$. \\\\
Enfin, on démontre un théorème sur la génération du groupe symétrique d'un groupe abélien $G$ avec des transpositions et des translations par des éléments du groupe $G$.

\section{Généralités sur les groupes abéliens finis}

On commence par rappeler quelques faits utiles sur les groupes abéliens finis. Pour ce qui est des notations, on notera $\abs{X}$ le cardinal d'un ensemble fini $X$, et parfois $\mbb Z_n$ pour $\mbb Z/n\mbb Z$. Enfin, si $G$ est un groupe, on notera $H\leq G$ pour signifier que $H$ est un sous-groupe de $G$ et $H<G$ si $H$ est un sous-groupe strict de $G$.

\subsection{Dualité}
Soit $G$ un groupe abélien fini. On note $\widehat G$ le groupe dual de $G$, c'est à dire le groupe des morphismes de $G$ dans $\cpx^\times$ (groupe multiplicatif du corps des nombres complexes), aussi appelés caractères. On rappelle que $G$ est isomorphe à un produit de groupes cycliques $\mbb Z_{d_1} \times \dots \times \mbb Z_{d_n}$ avec $d_1>1$ et $d_1\mid d_2\mid \dots \mid d_n$, et qu'il y a unicité de ces coefficients (appelés facteurs invariants). Le dual d'un produit de groupes abéliens finis est le produit des duaux, et le dual de $\mbb Z_n$ est isomorphe à $\mbb Z_n$, en choisissant une racine primitive $n$-ème de l'unité. On en déduit :
$$G\cong \dual G$$
Cependant, il n'y a pas d'isomorphisme canonique entre ces deux groupes en général. Notons qu'un morphisme $f:G\longrightarrow H$ induit un morphisme $f^* :\dual{H} \longrightarrow \dual{G}$ défini par $f^*(\chi)=\chi \circ f$. Cela définit un foncteur contravariant de la catégorie des groupes abéliens finis dans elle même. \\

\begin{prop}
    (Bidualité) \normalfont Soit $G$ un groupe abélien fini. On dispose d'un isomorphisme :
    $$G\arw{\alpha_G} \dual{\dual G}$$
    qui à $x$ associe le caractère $\chi \mapsto \chi(x)$. Cet isomorphisme est naturel en $G$, au sens où, pour tout morphisme $f:G\longrightarrow H$ entre groupes abéliens finis, le diagramme suivant commute :
    \begin{center}
\begin{tikzcd}
G \arrow[r, "f"] \arrow[d, "\alpha_G"']    & H \arrow[d, "\alpha_H"] \\
\widehat{\widehat{G}} \arrow[r, "f^{**}"'] & \widehat{\widehat{H}}  
\end{tikzcd}
    \end{center}
    En conséquence, \textbf{la catégorie des groupes abéliens finis est équivalente à sa duale}. Cela induit notamment une correspondance entre sous-groupes et quotients.
\end{prop}

\begin{proof}
    $\alpha_G$ est clairement un morphisme bien défini de $G$ vers $\widehat{\widehat G}$. Le diagramme commute : soit $x\in G$ et $\chi \in \dual H$ :
    $$f^{**}(\alpha_G(x))(\chi)=(\alpha_G(x) \circ f^*)(\chi)=f^*(\chi)(x)=\chi\circ f(x)$$
    et :
    $$\alpha_H(f(x))(\chi)=\chi(f(x))$$
\end{proof}

\begin{defin}
    (Anhilateur et noyau) \normalfont
    Pour $H\leq G$, on note $H^\bot$ l'anhilateur de $H$, c'est à dire le sous-groupe de $\dual G$ des caractères nuls sur $H$. Pour $K\leq \dual G$, on note $K^\top$ le noyau de $K$, c'est à dire l'intersection des noyaux des $\chi \in K$.
\end{defin}

\begin{prop}
\normalfont
    Soient $H\leq G$ et $K\leq \dual G$. On a les isomorphismes canoniques suivants :
    $$\boxed{\widehat{G/H}\cong H^\bot}$$
    et :
    $$\boxed{K^\top \cong \widehat{\dual G /K}}$$
    En particulier : $\abs{G}=\abs{H}\times \abs{H^\bot}=\abs{K}\times\abs{K^\top}$. On a de plus la compatibilité suivante : $(H^\bot)^\top = H$ et $(K^\top)^\bot = K$. Enfin, si l'on considère que $(H^\bot)^\bot$ est un sous-groupe de $\dual{\dual G}$, alors :
    $$\alpha_G(H)=H^{\bot \bot}$$
    On a une relation similaire avec $\top$. Notons que $\bot$ et $\top$ définissent une bijection des sous-groupes de $G$ vers les sous-groupes de $\dual G$.
\end{prop}

\begin{proof}
    D'abord, $\widehat{G/H}=\operatorname{Hom}(G/H,\cpx^\times)\cong \{f\in \operatorname{Hom}(G,\cpx^\times)\mid H\subseteq \operatorname{Ker} f\}=H^\bot$ par propriété universelle du quotient. Ensuite : $\dual{\dual G/K}=\Hom(\dual G/K,\cpx^\times)\cong\{f\in \dual{\dual G} \mid K\subseteq \operatorname{Ker}f\}\cong \{x\in G\mid \forall \chi \in K\  \chi(x)=0\}=K^\top$. Les égalités sur les cardinaux s'en déduisent directement puisqu'un groupe abélien fini est isomorphe à son dual. Ensuite on a clairement :
    $$H^{\bot \top}\supseteq H$$ 
    et l'autre inclusion est vraie pour des raisons de cardinal. On raisonne de même pour la deuxième égalité. Enfin, on a aisément $\alpha_G(H)\subseteq H^{\bot \bot}$ et l'autre inclusion est également vraie pour des raisons de cardinal.
\end{proof}

\subsection{Correspondance entre sous-groupes et quotients}

\begin{prop}
     \normalfont Soit $G$ un groupe abélien fini. On note $\mc S(G)$ l'ensemble des sous-groupes de $G$. Il existe $\Gamma$ une bijection décroissante de réciproque décroissante(pour l'inclusion), de $\mc S(G)$ vers $\mc S(G)$ telle que, pour tout $H\leq G$ :
    $$\Gamma(H)\cong G/H$$
    et
    $$\Gamma^{-1}(H)\cong G/H$$
    On appellera \textbf{division} une telle bijection $\Gamma$. Une division donne ainsi une correspondance entre les sous-groupes de $G$ et les quotients de $G$.
\end{prop}

\begin{proof}
    Choisissons $\theta$ un isomorphisme de $G$ vers $\dual G$. Un tel isomorphisme induit clairement une bijection $\Theta$ des sous-groupes de $G$ vers les sous-groupes de $\dual G$. On pose alors $\Gamma$ la composée :
    \begin{center}
    \begin{tikzcd}
    \mathcal{S}(G) \arrow[r, "\Theta"] & \mathcal{S}(\widehat G) \arrow[r, "\bullet^\top"] & \mathcal{S}(G)
    \end{tikzcd}
    \end{center}
    C'est une composée de deux bijections, l'une croissante et l'autre décroissante, donc c'est une bijection décroissante (et de même pour la réciproque). Il reste à constater :
    $$\Gamma(H)=\Theta(H)^\top\cong \dual{\dual G / \Theta(H)}\cong \dual G/\Theta(H)= \theta(G)/\theta(H) \cong G/H$$
    et 
    $$\Gamma^{-1}(H)=\Theta^{-1}(H^\bot)\cong H^\bot \cong \widehat{G/H}\cong G/H$$
    A priori, ces isomorphismes ne sont pas canoniques, même une fois $\Gamma$ fixé.
\end{proof}

Une conséquence intéressante est le fait suivant :

\begin{prop}
    \normalfont Soit $G$ un groupe abélien fini. On note $\min G$ le cardinal minimal d'une partie génératrice de $G$. On rappelle que $\min G/H\leq \min G$ (si $x_1,\dots,x_n$ génèrent $G$, leurs images génèrent $G/H$). Ainsi, par la correspondance entre sous-groupes et quotients, on a gratuitement que pour tout $H\leq G$ :
    $$\min H\leq \min G$$
\end{prop}

\subsection{Sous-groupes isomorphes à un groupe fixé}

On utilisera le lemme suivant à plusieurs reprises :

\begin{lem}
    \normalfont Soient $A$ et $B$ deux groupes abéliens finis. On note $\op{Sub}_B(A)$ l'ensemble des sous-groupes de $A$ isomorphes à $B$ et $\op{Mono}(B,A)$ l'ensemble des morphismes injectifs (ou monomorphismes) de $B$ dans $A$. On a alors :
    $$\boxed{\abs{\op{Sub}_B(A)} = \frac{\abs{\op{Mono}(B,A)}}{\abs{\op{Aut}B}}}$$
\end{lem}

\begin{proof}
    Pour le voir, il suffit de remarquer que $\op{Aut}B$ agit librement sur $\op{Mono}(B,A)$ et que les orbites de l'action s'identifient aux classes d'isomorphisme de monomorphismes de $B$ vers $A$, c'est à dire aux sous-groupes de $A$ isomorphes à $B$.
    \\\\
    Plus précisément, si $g\in \op{Aut}B$ et $f\in \op{Mono}(B,A)$, l'action de $g$ sur $f$ est donnée par $f\cdot g = f\circ g$ (action à droite).
\end{proof}

\section{Fonctions Abéliennes}

\subsection{Algèbre des fonctions Abéliennes}   
On considère un système de représentants à isomorphisme près des groupes abéliens finis, $\mbb G$. Pour rester dans la théorie des ensembles, on peut prendre les sous-groupes abéliens des groupes symétriques, mais on ne se préoccupera pas de ce type de questions. On note $\mbb A$ le $\mbb C$-espace vectoriel des applications de $\mbb G$ dans $\mbb C$. Ces applications sont appelées \textbf{fonctions abéliennes}. On munit $\mbb A$ du produit de convolution défini comme cela : soient $f,g\in \mbb A$, on définit :
$$\boxed{f*g(G)=\sum_{H\leq G}f(H)g(G/H)}$$
pour $G\in \mbb G$ (la somme porte sur tous les sous-groupes de $G$, pas seulement à isomorphisme près). Ici, il faut comprendre $f(G)$ comme $f(G_0)$ avec $G_0$ le représentant de la classe d'isomorphisme de $G$. On peut d'ailleurs voir les éléments de $\mbb A$ comme des "applications" qui à un groupe fini abélien $G$ associent un nombre indépendant de $G$ à isomorphisme près. On définit aussi un produit terme à terme :
$$f\cdot g(G)=f(G)g(G)$$
Notons d'ailleurs que $fg$ et $f*g$ sont bien définis puisque leur valeur en $G$ ne dépend pas du choix de $G$ à isomorphisme près. On remarque que $\delta$, la fonction abélienne valant $1$ sur le groupe trivial et $0$ pour tout autre groupe, est un élément neutre pour $*$.

\begin{defin} (L'algèbre $\mbb A$)
\normalfont \\
    $(\mbb A, *)$ est une $\mbb C$-algèbre commutative, associative et unitaire. \textbf{Dans la suite, $\mbb A$ désignera la $\mbb C$-algèbre $\mbb A$ munie de la loi $*$}. \\
    Ses éléments inversibles sont exactement les fonctions $f$ telles que $f(1)\neq 0$. $\mbb A$ est donc un anneau local.
\end{defin}

\begin{proof}
    Fixons $G\in \mbb G$. \\
    Choisissons une division $\mc{S}(G)\arw{\Gamma}\mc S(G)$ sur $G$. Voyons la commutativité :
    $$f*g(G)=\sum_{H\leq G}f(H)g(G/H)=\sum_{K\leq G}f(\Gamma^{-1}(K))g(K)=\sum_{K\leq G}f(G/K)g(K)=g*f(G)$$
    en posant $K=\Gamma(H)$ (changement de variable bijectif). \\ 
    À présent, voyons l'associativité :
    \begin{align*}
    f*(g*h)(G)&=\sum_{H\leq G}f(H)(g*h)(G/H)
    \\&=\sum_{H\leq G}\sum_{K\leq G/H}f(H)g(K)h(G/H/K)
    \\&= \sum_{H\leq G}\sum_{H\leq L\leq G}f(H)g(L/H)h\left(\frac{G/H}{L/H}\right)
    \\&= \sum_{H\leq G}\sum_{H\leq L\leq G} f(H)g(L/H)h(G/L)
    \end{align*}
    et :
    \begin{align*}
    (f*g)*h(G)&=\sum_{L\leq G}(f*g)(L)h(G/L)
    \\&=\sum_{L\leq G}\sum_{H\leq L}f(H)g(L/H)h(G/L)
    \end{align*}
    $\delta$ est le neutre pour $*$ : par commutativité, il suffit de vérifier $f*\delta=f$, ce qui est clair :
    $$f*\delta(G)=\sum_{H\leq G}f(H)\delta(G/H)=f(G)$$
    Ensuite, si $f(1)\neq 0$, on peut construire par récurrence sur l'ordre de $G$ un nombre $g(G)$ qui ne dépend que de $G$ à isomorphisme près : on pose $g(1)=1/f(1)$, et pour tout groupe $G$ non trivial :
    $$\boxed{g(G)=-\frac 1 {f(1)}\sum_{H<G}g(H)f(G/H)}$$
    qui est bien défini par récurrence forte (le membre de droite ne dépend pas de $G$ à isomorphisme près car c'est le cas des $g(H)$ pour $H<G$). $g$ définit donc une fonction abélienne et on vérifie aisément (par récurrence forte) que :
    $$f*g=g*f=\delta$$
    L'ensemble des éléments non inversibles est donc l'idéal maximal formé des $f$ nulles en $1$, c'est donc le seul idéal maximal de $\ring$.
\end{proof}

\subsection{Fonctions multiplicatives}

\begin{lem}
    \normalfont (Sous-groupe d'un produit de groupes d'ordres premiers entre eux) \\
    Soient $G$ et $H$ deux groupes finis de cardinaux $m$ et $n$ premiers entre eux. Alors les sous-groupes de $G\times H$ sont exactement les produits $A\times B$ avec $A\leq G$ et $B\leq H$, et on a ainsi une bijection :
    $$\boxed{\mc S(G)\times \mc S(H)\arw{\sim} \mc S(G\times H)}$$
\end{lem}

\begin{proof}
    Soit $K\leq G\times H$. On pose $A=\pi_G(K)$ et $B=\pi_H(K)$. On se donne une relation de Bézout $1=mu+nv$. Soit $(a,b)\in A\times B$. Il existe $c\in H$ et $d\in G$ tels que : $(a,c)\in K$ et $(d,b)\in K$. On a donc :
    $$(a,c)^{nv}(d,b)^{mu}=(a,b)$$
    par théorème de Lagrange, donc $(a,b)\in K$. De plus, $K$ est clairement contenu dans $A\times B$, donc :
    $$K=A\times B$$
    Ainsi, l'application $\mc S(G)\times \mc S(H)\longrightarrow \mc S(G\times H)$ qui à $(A,B)$ associe $A\times B$ est surjective, et elle est injective car $A=\pi_G(A\times B)$ et $B=\pi_H(A\times B)$.
\end{proof}

\begin{defin} (Fonctions multiplicatives) \normalfont
    Une fonction abélienne $f$ est dite \textbf{multiplicative} si pour tous $G,H\in \mbb G$ d'ordres premiers entre eux, on a : $f(G\times H)=f(G)f(H)$ et si $f(1)=1$. On note $\mbb M$ l'ensemble des fonctions abéliennes multiplicatives, c'est un sous-groupe de $\ring^\times$. $f$ est dite \textbf{complétement multiplicative} si la relation reste valable pour $G$ et $H$ quelconques.
\end{defin}

\begin{proof}
    D'abord, $\mbb M\subseteq \ring^\times$ d'après la proposition qui précède. Ensuite, le produit de deux fonctions abéliennes multiplicatives est multiplicative : si $f$ et $g$ sont multiplicatives, on a $f*g(1)=f(1)g(1)=1$ et pour $\abs{G}\wedge \abs{H}=1$ :
    $$f*g(G\times H)=\sum_{K\leq G\times H}f(K)g((G\times H)/K)=\sum_{A\leq G,\ B\leq H}f(A\times B)g(G/A \times H/B)$$
    par le lemme précédent. Or $A$ et $B$ ont des ordres premiers entre eux (par Lagrange) et pareil pour $G/A$ et $H/B$, donc, par multiplicativité de $f$ et $g$ :
    $$f*g(G\times H)=\sum_{A\leq G,\ B\leq H}f(A)f(B)g(G/A)g(H/B)=(f*g(G))(f*g(H))$$
    donc $f*g$ est multiplicative. Voyons maintenant que $f^{-1}$ est multiplicative. Pour cela, on montre par récurrence forte sur $\abs{G}\times \abs{H}$ que, lorsque $\abs{G}\wedge \abs{H}=1$ : $f^{-1}(G\times H)=f^{-1}(G)f^{-1}(H)$. Si $G$ est trivial ou si $H$ est trivial, c'est clair. Supposons $G$ et $H$ non triviaux. Par hypothèse de récurrence on peut écrire :
    \begin{align*}
        0 &= \delta(G\times H) \\
        &= \sum_{\substack{A\leq G\\ B\leq H}}f^{-1}(A)f^{-1}(B)f(G/A)f(H/B) \\
        &= \sum_{\substack{A< G\\ B< H}}f^{-1}(A)f^{-1}(B)f(G/A)f(H/B) + f^{-1}(G\times H)\\
        &\ \ \ \ \ +\sum_{A< G}f^{-1}(A)f^{-1}(H)f(G/A)+\sum_{B< H}f^{-1}(B)f^{-1}(G)f(H/B) \\
        &= \sum_{A<G}f^{-1}(A)f(G/A)\sum_{B<H}f^{-1}(B)f(H/B) + f^{-1}(G\times H) \\
        &\ \ \ \ \ - f^{-1}(H)f^{-1}(G)-f^{-1}(G)f^{-1}(H) \\
        &=(-1)^2f^{-1}(G)f^{-1}(H)-2f^{-1}(G)f^{-1}(H)+f^{-1}(G\times H)
    \end{align*} 
    donc $f^{-1}(G\times H)=f^{-1}(G)f^{-1}(H)$, ce qui achève la récurrence. $\mbb M$ est donc un sous-groupe de $\ring^\times$ ($\delta$ est clairement multiplicative).
\end{proof}

\subsection{Lien avec la convolution de Dirichlet}
\begin{defin}
    (Fonctions arithmétiques sur $\mbb N^*$) \normalfont
    On peut aussi définir $\mbb A_{\mbb N^*}$ comme la $\mbb C$-algèbre des fonctions de $\mbb N^*$ dans $\cpx$ avec le produit de convolution $f*g(n)=\sum_{d\mid n}f(d)g(n/d)$.
    On définit de même les fonctions multiplicatives $\mbb M_{\mbb N^*}$ (ce sont les fonctions arithmétiques qui vérifient $f(1)=1$ et $f(ab)=f(a)f(b)$ dès que $a\wedge b=1$). Notons que $\mbb A_{\nat ^*}$ est un anneau intègre local.
\end{defin}

\begin{prop}
    \normalfont On dispose d'un morphisme surjectif de $\mbb C$-algèbres :
    $$\mbb A \longrightarrow \mbb A_{\mbb N^*}$$
    qui envoie $f$ sur $n\mapsto f(\mbb Z/n\mbb Z)$. \\
    Le noyau est l'idéal premier des fonctions abéliennes nulles sur les groupes cycliques. Ce morphisme induit un morphisme surjectif de groupes abéliens :
    $$\mbb M \longrightarrow \mbb M_{\mbb N^*}$$
\end{prop}

\begin{proof}
    On vérifie facilement que c'est un morphisme d'algèbres car les sous-groupes (et les quotients) de $\mbb Z/n\mbb Z$ sont en correspondance bijective avec les diviseurs de $n$. La surjectivité est claire, et le noyau est un idéal premier puisque $\mbb A_{\mbb N^*}$ est intègre. \\
    Ensuite, si $f\in \mbb M$, alors son image dans $\ring_{\nat ^*}$ est multiplicative, car si $m\wedge n=1$, $\mbb Z/m\mbb Z \times \mbb Z/n\mbb Z \cong \mbb Z/(mn)\mbb Z$. \\
    La surjectivité de ce morphisme est encore vraie : soit $f\in \mbb M_{\nat ^*}$. On définit simplement, pour $G\in \mbb G$ : $g(G)=f(n)$ si $G$ est cylique d'ordre $n$, et $0$ si $G$ n'est pas cyclique. On vérifie facilement que $g$ est multiplicative.
\end{proof}

\subsection{Exemples}
Donnons à présent quelques exemples importants de fonctions abéliennes. 

\begin{defin} \normalfont
    La fonction $1$ (valant constamment $1$) est multiplicative, donc d'inverse multiplicatif. \textbf{On note $\mu$ cet inverse (fonction de Mobïus abélienne)}. D'après la proposition qui précède sur le lien avec la convolution de Dirichlet, on a $\mu(\mbb Z/n\mbb Z)=\mu(n)$ pour tout $n\in \nat^*$. Dans la partie suivante, on donne une formule explicite pour $\mu(G)$ pour un groupe abélien fini $G$.
\end{defin}

 On note aussi $\varphi(G)$ le nombre de générateurs de $G$. Encore une fois, on a $\varphi(\mbb Z/n\mbb Z)=\varphi(n)$. La fonction $\operatorname{Card}$ est clairement multiplicative et induit la fonction identité de $\mbb N^*$ dans $\mbb M_{\nat ^*}$. La fonction \textbf{nombre de sous-groupes} est simplement $1*1$ (c'est aussi la fonction \textbf{nombre de quotients}).

\begin{prop}
    \normalfont $\varphi$ est multiplicative et $\varphi * 1=\operatorname{Card}$, i.e. $\varphi = \mu * \operatorname{Card}$.
\end{prop}

\begin{proof}
    Soit $G\in \mbb G$. On regroupe les éléments de $G$ selon le sous-groupe qu'ils engendrent :
    $$\abs{G}=\sum_{H\leq G}\varphi(H)$$
    donc $\varphi * 1 = \operatorname{Card}$ et $\varphi = \mu * \operatorname{Card}\in \mbb M$ car $\mbb M$ est un groupe. Notons qu'à l'aide du morphisme défini précedemment, on en déduit aussi la multiplicativité de la fonction d'Euler.
\end{proof}

Par le même procédé, on démontre que le nombre de $t$-uplets $(a_1,\dots,a_t)$ générant $G$ donne la fonction multiplicative $\mu * (G\mapsto \abs{G}^t)$. La fonction $\mu$ intervient ainsi dans de nombreux calculs. On peut aussi s'intéresser au nombre de parties à $d$ éléments qui engendrent $G$ et obtenir $\mu * {\abs{\bullet} \choose d}$. 
\\
Enfin, la fonction $N_t = \mu * t^{\abs{\bullet}}$ pour $t\geq 1$ nous interessera dans la section suivante, où on verra que $\abs{G}\mid N_t(G)$.

\begin{prop}
\normalfont
    $N_1$ est la fonction $\delta$ et $N_2$ est le nombre de parties génératrices de $G$.
\end{prop}

\begin{proof}
    On a $N_1=\mu * 1=\delta$. Notons ensuite $P$ la fonction "nombre de parties génératrices". Pour tout groupe $G\in \mbb G$, on peut dénombrer les parties de $G$ en les regroupant selon le sous-groupe $H\leq G$ qu'elles engendrent :
    $$2^{\abs{G}}=\sum_{H\leq G}P(H)$$
    Ainsi $2^{\abs{\bullet}}=P*1$ donc $P=N_2$.
\end{proof}

\subsection{Calcul de $\mu$}
Dans cette partie, on donne une formule explicite pour $\mu(G)$ (où $\mu * 1 =\delta$) en fonction des facteurs invariants de $G$. Pour cela, $\mu$ étant multiplicative, il est clair qu'il suffit de la calculer pour les $p$-groupes.

\begin{prop}
    \normalfont (Cas des espaces vectoriels sur $\mbb F_p$)
    Soit $p$ un nombre premier et $n\in \nat$. On a :
    $$\boxed{\mu\left(\mathbb Z_p^n\right)=(-1)^np^{\frac{n(n-1)}{2}}}$$
    En particulier, cet exemple montre que $\mu$ n'est pas bornée (contrairement à la fonction $\mu$ de Möbius usuelle).
\end{prop}

\begin{proof}
    Les sous-groupes de $\mbb Z_p^n$ sont exactement les sous $\mathbb F_p$-espaces vectoriels de $\mbb Z_p^n$, ils sont donc de la forme (à isomorphisme près) $\mbb Z_p^d$ avec $d$ leur dimension. Le nombre de sous-groupes de $\mbb Z_p^n$ isomorphes à $\mbb Z_p^d$, pour $0\leq d\leq n$ est donné par :
    $$\frac{\abs{\op{Mono}(\mbb Z_p^d,\mbb Z_p^n)}}{\abs{\op{Aut}\mbb Z_p^d}}=
    \frac{(p^n-1)\dots(p^n-p^{d-1})}{(p^d-1)\dots(p^d-p^{d-1})}$$
    (voir section $2$). \\
    Ceci étant dit, il est clair qu'il suffit de montrer que pour tout $n$ :
    $$\sum_{d=0}^{n}(-1)^dp^{d(d-1)/2}\frac{(p^n-1)\dots(p^n-p^{d-1})}{(p^d-1)\dots(p^d-p^{d-1})}=\delta(n)$$
    (par récurrence forte, cette égalité donne le résultat voulu)
    \\
    Notons $A_n$ le membre de gauche. Clairement $A_0=1$ (le produit est vide). On a, pour $n\geq 1$  :
    \begin{align*}
        A_n &= \sum_{d=0}^n (-1)^d p^{0}p^1\dots p^{d-1} \frac{(p^n-1)\dots(p^n-p^{d-1})}{(p^d-1)\dots(p^d-p^{d-1})} \\
        &= \sum_{d=0}^n (-1)^d  \frac{(p^n-1)\dots(p^n-p^{d-1})}{(p^d-1)\dots(p-1)} \\
        &= \frac 1 D \sum_{d=0}^n (-1)^d  (p^{d+1}-1)\dots (p^n-1)\times(p^n-1)\dots(p^n-p^{d-1}) \\
        &= \frac 1 D \sum_{d=0}^n \prod_{i=0}^{d-1} (p^i-p^n)  \prod_{i=d+1}^n (p^i-1) 
    \end{align*}
    où $D$ est le dénominateur commun $(p-1)\dots (p^n-1)$. À présent, montrons par récurrence que pour tout $k$ entre $0$ et $n$ :
    $$\sum_{d=0}^k \prod_{i=0}^{d-1} (p^i-p^n)  \prod_{i=d+1}^n (p^i-1)  =\prod_{i=1}^{k} (p^i-p^n)  \prod_{i=k+1}^n (p^i-1) $$
    Pour $k=0$ le résultat est clair. Supposons l'énoncé vrai au rang $k<n$ et montrons qu'il est encore vrai au rang $k+1$ :
    \begin{align*}
        \sum_{d=0}^{k+1} \prod_{i=0}^{d-1} (p^i-p^n)  \prod_{i=d+1}^n (p^i-1)  &=\prod_{i=1}^{k} (p^i-p^n)  \prod_{i=k+1}^n (p^i-1) + \prod_{i=0}^{k} (p^i-p^n)  \prod_{i=k+2}^n (p^i-1) \\
        &= \prod_{i=1}^{k} (p^i-p^n) \prod_{i=k+2}^n (p^i-1) \times \left( p^{k+1}-1 + 1-p^n\right) \\
        &= \prod_{i=1}^{k} (p^i-p^n) \prod_{i=k+2}^n (p^i-1) \times \left( p^{k+1}-p^n\right) \\
        &= \prod_{i=1}^{k+1} (p^i-p^n) \prod_{i=k+2}^n (p^i-1) 
  \end{align*}
    ce qui achève la récurrence. Au rang $k=n$ on obtient :
    $$A_n = \frac 1 D \prod_{i=1}^n(p^i-p^n)=0=\delta(n)$$
    car $n\geq 1$. 
\end{proof}

A priori, le calcul précédent ne suffit pas à obtenir $\mu(G)$ en général. Heureusement, pour tous les autres $p$-groupes, $\mu$ se révèle être nulle.

\begin{prop}
    \normalfont Soit $G$ un $p$-groupe abélien non élémentaire (cela signifie qu'il existe un élément d'ordre $p^k$ avec $k\geq 2$). On a :
    $$\mu(G)=0$$
\end{prop}

\begin{proof}
    On le montre par récurrence forte sur $\abs G$. Supposons que c'est vrai pour tout groupe $p$-abélien non élémentaire de cardinal strictement inférieur à $\abs G$ (il n'y a pas besoin d'initialiser). On a alors :
    $$\mu(G)=-\sum_{H<G}\mu(H)$$
    Par  hypothèse de récurrence, seuls les sous-groupes élémentaires contribuent à cette somme. On note $G(p)$ le sous-groupe de $p$-torsion de $G$, et on a donc :
    $$\mu(G)=-\sum_{H\leq G(p)}\mu(H)$$
    car $G(p)<G$ puisque $G$ n'est pas élémentaire. Au total :
    $$\mu(G)=-\mu*1(G(p))=-\delta(G(p))=0$$
    puisque $G(p)\neq 0$.
\end{proof}

On peut résumer ces deux observations ainsi :

\begin{theorem}
    \normalfont
    Si $G$ est produit de $p$-groupes élémentaires, on note $\dim_p G$ la puissance à laquelle apparaît $\mbb Z_p$ dans la décomposition de $G$ en produit de $p$-groupes élémentaires, et on a :
    $$\boxed{\mu(G)=\prod_{p\in \mbb P} (-1)^{\dim_p G}p^{\frac{\dim_p G (\dim_p G -1)}2}}$$
    avec $\mathbb P$ l'ensemble des nombres premiers. 
    \\
    Dans le cas contraire, $\mu(G)=0$.
\end{theorem}

\begin{proof}
    On l'obtient directement avec la multiplicativité de $\mu$ et la décomposition en $p$-Sylows : $G\cong \bigoplus_{p\in \mathbb P}\bigoplus_{k\geq 1}\left(\mbb Z_p^k\right)^{n_{p,k}}$.
\end{proof}

\begin{defin}
    \normalfont
    On dira que $G$ est \textbf{élémentaire} s'il est produit (fini) de $p$-groupes élémentaires. Les groupes élémentaires sont exactement les groupes qui ont une valeur de $\mu$ non nulle. Le sous-ensemble de $\mbb G$ des groupes élémentaires est alors naturellement en bijection avec $\nat^*$, via $G\mapsto \abs{G}$ (deux groupes élémentaires sont isomorphes si et seulement si ils ont même cardinal). Cet ensemble est aussi stable par produit, sous-groupe et quotient, (tout comme le sous-ensemble de $\mbb G$ constitué des groupes cycliques, eux aussi entièrement déterminés par leur cardinal). Ainsi, pour un groupe élémentaire $G$, $\mu(G)$ ne dépend que du cardinal de $G$.
\end{defin}

\subsection{Applications}

Étant donnés deux groupes abéliens finis $A$ et $B$, on note $\op{Mono}(A,B)$ l'ensemble des morphismes injectifs de $A$ dans $B$ et $\op{Epi}(A,B)$ l'ensemble des morphismes surjectifs de $A$ dans $B$ (ces notions coïncident avec les notions de monomorphismes et épimorphismes dans la catégorie des groupes abéliens finis). La catégorie des groupes abéliens finis étant équivalente à sa duale, il y a autant de morphismes de $A$ vers $B$ que de morphismes de $B$ vers $A$, et les quantités $\abs{\op{Mono}(A,B)}$ et $\abs{\op{Epi}(B,A)}$ sont égales. 

\begin{prop}
    \normalfont On dispose des relations suivantes :
    $$\abs{\op{Hom}(A,B)}=\abs{\op{Hom}(B,A)}=\sum_{H\leq A}\abs{\op{Mono}\left(\frac A H,B\right)}=\sum_{H\leq B}\abs{\op{Epi}(A,H)}$$
    Par commutativité de $*$, on peut aussi écrire ça $\sum_{H\leq A}\abs{\op{Mono}\left(H,B\right)}$. \\
    Par la formule $\mu*1=\delta$ on en déduit immédiatement :
    $$\abs{\op{Mono}(A,B)}=\abs{\op{Epi}(B,A)} = \sum_{H\leq A}\mu(A/H)\abs{\op{Hom}(H,B)}=\sum_{H\leq B}\mu(B/H)\abs{\op{Hom}(A,H)}$$
\end{prop}

\begin{proof}
    On dénombre les morphismes de $A$ vers $B$ en les classant selon leur noyau, qui peut être n'importe quel sous-groupe (distingué) de $A$ :

    $$\abs{\op{Hom}(A,B)}=\sum_{H\leq A}\abs{\{f\in \op{Hom}(A,B)\mid \op{Ker}f=H\}}=\sum_{H\leq A}\abs{\op{Mono}(A/H,B)}$$
    par propriété universelle du quotient. Pour la formule avec les épimorphismes, il s'agit cette fois de dénombrer les morphismes de $A$ vers $B$ en les classant selon leur image (ou selon leur conoyau).
\end{proof}

On en déduit une formule pour le nombre de sous-groupes de type donné (on dit qu'un sous-groupe $H$ de $A$ est de type $B$ s'il est isomorphe à $B$).

\begin{prop}
    \normalfont
    Soient $A,B$ deux groupes abéliens finis. Le nombre de sous-groupes de $A$ isomorphes à $B$ est :
    $$ \boxed{\abs{\op{Sub}_B(A)}=\frac{\displaystyle \sum_{H\leq B}{\mu(B)\abs{\op{Hom}(B/H,A)}}}{\displaystyle \sum_{H\leq B}{\mu(B)\abs{\op{Hom}(B/H,B)}}}}$$
\end{prop}

\begin{proof}
    On utilise la formule générale :
    $$\abs{\op{Sub_B(A)}}=\frac{\abs{\op{Mono}(B,A)}}{\abs{\op{Aut}B}}$$

    Or, $B$ étant fini, on a naturellement $\op{Aut}B=\op{Mono}(B,B)$. Il ne reste plus qu'à appliquer les formules qui précèdent.
\end{proof}

\begin{remark}
    \normalfont D'après le calcul de $\mu$, si $B$ est un $p$-groupe, on peut restreindre les sommes aux sous-espaces vectoriels de $B(p)$ (la $p$-torsion de $B$). La formule est alors assez efficace si le groupe $B$ est suffisamment petit pour que l'on puisse calculer les quotients $B/H$ présents dans la formule pour tous les sous-espaces vectoriels $H$ de $B$. Le calcul du cardinal de $\Hom$ est aisé puisque $\abs{\Hom}$ est multiplicatif en chaque variable.
\end{remark}

On propose maintenant une démonstration du théorème de simplification des groupes finis (\textbf{dans le cas abélien seulement}) adaptée de \cite{td3}.

\begin{lem}
    \normalfont (Yoneda numérique) \\
    Soient $A,B$ deux groupes abéliens finis tels que pour tout $X$ un groupe abélien fini, on ait :
    $$\abs{\op{Hom}(A,X)}=\abs{\op{Hom}(B,X)}$$
    Alors $A$ et $B$ sont isomorphes. Ce lemme reste vrai pour des groupes finis non nécessairement commutatifs mais la convolution ne suffit plus à l'établir (voir \cite{lovasz} pour une démonstration dans ce cadre). \\\\
    De plus, il suffit que cette égalité soit vérifiée pour tout \textbf{groupe cyclique} $X$ (ou encore pour tout $p$-groupe, pour tout $p$ premier).
\end{lem}

\begin{proof}
    Constatons d'abord que, pour tout groupe abélien fini $X$, on a $\abs{\op{Mono}(A,X)}=\abs{\op{Mono}(B,X)}$. Il suffit pour cela d'utiliser la formule :
    $$\abs{\op{Mono}(A,X)}=\sum_{H\leq X}\mu(X/H)\abs{\op{Hom}(A,H)}$$
    et d'utiliser l'hypothèse du lemme pour remplacer le $A$ par un $B$ dans la formule. Comme pour le lemme de Yoneda, on applique cette relation à un $A$ et à $B$ : $\op{Mono}(A,A)$ n'est pas vide donc $\op{Mono}(B,A)$ n'est pas vide, et réciproquement $\op{Mono}(A,B)$ n'est pas vide. Puisque ce sont des groupes finis, on en déduit successivement $\abs{B}\leq \abs{A}$ et $\abs{A}\leq \abs{B}$ donc $A$ et $B$ ont même cardinal, or il existe un sous-groupe de $A$ isomorphe à $B$, et par cardinalité ce sous-groupe est $A$. $A$ et $B$ sont donc isomorphes. Il suffit de vérifier cela pour tout groupe cyclique ou pour tout $p$-groupe puisqu'un groupe abélien fini est produit de tels groupes (et en utilisant la propriété universelle du produit).
\end{proof}

\begin{theorem}     \normalfont (Simplification des Groupes Abéliens Finis) \\
    Si $A,B,C$ sont trois groupes abéliens finis vérifiant $A\times B \cong A\times C$, alors $B$ et $C$ sont isomorphes.
\end{theorem}

\begin{proof}
    On utilise la propriété universelle du coproduit dans la catégorie des groupes abéliens (finis) :
    \\
    Soit $X$ un groupe abélien fini quelconque, on a :
    $$\abs{\op{Hom}(A,X)}\times \abs{\op{Hom}(B,X)}=\abs{\op{Hom}(A\times B,X)}=\abs{\op{Hom}(A\times C,X)}=\abs{\op{Hom}(A,X)}\times \abs{\op{Hom}(C,X)}$$
    Aucun de ces facteurs n'est nul, donc on obtient :
    $$\abs{\op{Hom}(B,X)}=\abs{\op{Hom}(C,X)}$$
    et on conclut par le lemme de Yoneda numérique : $B$ et $C$ sont isomorphes.
\end{proof}

Voici une autre conséquence intéressante :

\begin{theorem}
    \normalfont Soient $A$ et $B$ deux groupes abéliens finis. Si pour tout $d\in \nat^*$, $A$ et $B$ ont autant d'éléments d'ordre $d$, alors ils sont isomorphes.
\end{theorem}

\begin{proof}
    L'hypothèse se traduit en :
    $$\forall d\in \nat^*\ \abs{\op{Mono}(\mbb Z/d\mbb Z,A)}=\abs{\op{Mono}(\mbb Z/d\mbb Z,B)}$$
    Par convolution (et parce que les sous-groupes des groupes cycliques sont cycliques) on obtient :
    $$\forall d\in \nat^*\ \abs{\op{Hom}(\mbb Z/d\mbb Z,A)}=\abs{\op{Hom}(\mbb Z/d\mbb Z,B)}$$
    On conclut alors par lemme de Yoneda numérique.
\end{proof}

\begin{conjecture}
    \normalfont Soient $A$ et $B$ deux groupes abéliens finis. Si pour tout $d\in \nat^*$, $A$ et $B$ ont autant de sous-groupes d'ordre $d$, alors ils sont isomorphes.
\end{conjecture}

\section{Dénombrement par les actions de groupes}

Dans cette section, on va démontrer le théorème suivant concernant la fonction abélienne $N_t=\mu * t^{\abs{\bullet}}$.
\begin{theorem} \normalfont
    Pour tout $G\in \mbb G$, $N_t(G)$ est divisible par le cardinal de $G$. En particulier, le \textbf{nombre de parties génératrices de $G$} est divisible par $\abs{G}$, puisque c'est $N_2$. 
\end{theorem}

Notons que pour un groupe cyclique, on obtient que $\sum_{d\mid n}\mu(d)t^{n/d}$ est divisible par $n$, résultat que l'on peut obtenir (pour $t$ une puissance de $p$) par un argument de dénombrement des polynômes irréductibles unitaires de degré $d$ dans $\mbb F_t$. 

\subsection{Actions libres}
Soit $G$ un groupe et $X$ un $G$-ensemble. On dit qu'un élément de $X$ est \textbf{libre} si son stabilisateur est trivial. On note $L(X)$ l'ensemble des éléments libres de $X$. On dit que $X$ est \textbf{libre} si tous ses éléments sont libres. \\

Remarquons que $L(X)$ est stable par l'action de $G$. C'est donc naturellement un $G$-ensemble libre. On dispose de la propriété arithmétique suivante :

\begin{prop}
\normalfont
    Si $X$ est un $G$-ensemble libre fini, alors $G$ est fini et le cardinal de $G$ divise le cardinal de $X$. Le quotient $\abs{X}/\abs{G}$ est le nombre d'orbites de $X$.
\end{prop}

\begin{corol}
\normalfont
    Si $X$ est un $G$-ensemble fini et si $G$ est fini, alors le cardinal de $L(X)$ est divisible par $\abs{G}$.
\end{corol}

\begin{proof}
    Prenons $x\in X$, puisque le stabilisateur de $x$ est trivial, on a une bijection entre $G$ et $G\cdot x$, donc $G$ est fini, et en partitionnant $X$ en orbites (toutes de taille $\abs{G}$), on obtient $\abs{G}\mid \abs{X}$.
\end{proof}

\subsection{Actions régulières}
On cherche à étudier un type bien particulier d'action d'un groupe $G$ : Fixons $X$ un ensemble et faisons agir $G$ sur $\mathcal F(G,X)$ (l'ensemble des applications de $G$ dans $X$) de la manière suivante :
$$g\cdot \alpha (x) = \alpha (xg)$$
pour tout $\alpha\in \mathcal F(G,X)$, tout $g\in G$ et tout $x\in X$. On vérifie aisément qu'il s'agit d'une action de groupe, qu'on appellera \textbf{action $X$-régulière}. 

\begin{remark} \normalfont Le cas $X=\{0,1\}$ correspond à l'action de $G$ sur l'ensemble de ses parties par translation (dans le mauvais sens). 
\end{remark}
Les actions régulières sont en quelque sorte universelles :
\begin{prop}
    (Plongement régulier) \normalfont Soit $X$ un $G$-ensemble. Il existe un morphisme injectif de $G$-ensembles de $X$ dans $\mathcal F(G,X)$. Autrement dit, tout $G$-ensemble se plonge dans un $G$ ensemble régulier.
\end{prop}

\begin{proof}
    Soit $x\in X$. On note $\varphi(x)$ la fonction $G\longrightarrow X$ qui à $g$ associe $gx$. Cela définit clairement une application injective $X\hookrightarrow \mathcal F(G,X)$ car $\varphi(x)(1)=x$. C'est un morphisme de $G$-ensembles : $\varphi(g\cdot x)(h)=hg\cdot x$ et $(g\cdot \varphi(x))(h)=\varphi(x)(hg)=hg\cdot x$.
\end{proof}

\begin{theorem} \normalfont
    Soit $H\trianglelefteq G$ un sous-groupe distingué de $G$. On considère l'action $X$-régulière sur $G/H$ et $\pi:G\longrightarrow G/H$ le morphisme quotient. Notons $\mathcal F(G,X)^H$ l'ensemble des $\alpha \in \mathcal F(G,X)$ fixés par tous les éléments de $H$ (autrement dit les $\alpha$ dont le stabilisateur contient $H$). On a alors une bijection naturelle :
    $$\boxed{\mathcal F(G,X)^H \cong \mathcal F(G/H,X)}$$
    De plus, en voyant naturellement $\mathcal F(G/H,X)$ comme un $G$ ensemble, la bijection est un \textbf{isomorphisme de $G$-ensembles}.
\end{theorem}

\begin{proof}
    On a un morphisme naturel $\mathcal F(G/H,X) \arw{\varphi}\mathcal F(G,X)^H$ défini par $\beta \mapsto \beta \circ \pi$. $\beta \circ \pi$ est bien fixée par $H$ : pour tout $h\in H$ et $x\in X$ on a $h\cdot \beta \pi (x)=\beta \pi (xh)=\beta\pi(x)$. Réciproquement, tout élément $\alpha \in \mathcal F(G,X)^H$ se factorise par $\pi$ car pour tout $h\in H$ et $x\in X$, $\alpha (xh)=h\cdot \alpha (x)=\alpha (x)$, ce qui permet de définir $\psi(\alpha)$ comme l'unique application faisant commuter le diagramme :
        \begin{center}
        \begin{tikzcd}
        G \arrow[r, "\alpha"] \arrow[d, "\pi"'] & X \\
        G/H \arrow[ru, "\psi(\alpha)"']         &  
        \end{tikzcd}
        \end{center}
    Le diagramme commute donc $\varphi\circ \psi(\alpha)=\alpha$ et on a clairement $\psi \circ \varphi(\beta)=\beta$ pour tout $\beta \in \mathcal F(G/H,X)$ par unicité de la factorisation. \\\\
    Enfin, il est clair que $\varphi$ est un morphisme de $G$-ensembles, d'où la conclusion.
\end{proof}

Dans la suite, on notera $\overline \alpha$ pour $\psi(\alpha)$.

\begin{corol}
    \normalfont Pour tout $\alpha\in \mathcal F(G,X)^H$, on a (en notant $\operatorname{Stab}$ pour le stabilisateur) :
    $$\boxed{\operatorname{Stab}_{G/H}(\bar\alpha)=\operatorname{Stab}_G(\alpha)/H}$$
\end{corol}

\begin{proof}
    Puisque on dispose d'un tel isomorphisme de $G$-ensembles entre $\mathcal F(G,X)^H$ et $\mathcal F(G/H,X)$, il y a une compatibilité aux stabilisateurs. Autrement dit, pour tout $\alpha\in \mathcal F(G,X)^H$, le stabilisateur de $\alpha$ est aussi le stabilisateur de $\bar \alpha$ en voyant $\mathcal F(G/H,X)$ comme un $G$-ensemble. On a donc :
    $$\operatorname{Stab}_{G}(\alpha)=\operatorname{Stab}_{G}(\bar \alpha)$$
    Or $\operatorname{Stab}_{G/H}(\bar \alpha)=\operatorname{Stab}_G(\bar \alpha)/H$ (clair). On en déduit la formule voulue.
\end{proof}
 
De ce qui précède, pour $H\trianglelefteq G$, on a une \textbf{correspondance bijective entre les éléments libres de $\mathcal F(G/H,X)$ et les éléments de $\mathcal F(G,X)$ dont le stabilisateur est $H$}. 

\begin{defin}
    \normalfont
    On notera, quand $G$ et $X$ sont finis, \textbf{$L_{\abs{X}}(G)$ le nombre d'éléments libres de $\mathcal F(G,X)$} (ça ne dépend que de $\abs{X}$ et de $G$ à isomorphisme près). Notons que ce nombre est \textbf{divisible par $\abs{G}$} d'après ce qui a été dit plus haut. 
\end{defin}

Avec cette notation, on a :
$$\boxed{L_{\abs{X}}(G/H)=\abs{\left\{\alpha \in \mathcal F(G,X) \mid \op{Stab}\alpha = H\right\}}}$$

\subsection{Calcul de $L_t(G)$ pour $G$ abélien}

On peut voir $L_t$ (la fonction définie précedemment) comme une fonction abélienne. On dispose alors d'une formule agréable pour ce nombre :

\begin{prop}
    \normalfont
    Soit $t\geq 1$. On a l'égalité de fonctions abéliennes suivante :
    $$\boxed{L_t = N_t}$$
    où $N_t=\mu * t^{\abs{\bullet}}$, 
\end{prop}

\begin{proof}
    On prend $X$ un ensemble à $t$ éléments, et on dénombre $\mathcal F(G,X)$ en regroupant les éléments selon leur stabilisateur :
    $$t^{\abs{G}}=\sum_{H\leq G}\abs{\{\alpha \in G \mid \operatorname{Stab}(\alpha)=H\}}
    =\sum_{H\leq G} L_t(G/H)$$
    On a donc $t^{\abs{\bullet}}=1*L_t$. On en déduit en convoluant par $\mu$ :
    $$L_t=\mu *t^{\abs{\bullet}}=N_t$$
\end{proof}

\begin{corol}
\normalfont
     Le théorème introduit en début de section en découle directement puisque $L_t(\abs{G})$ est divisible par $\abs{G}$ : \textbf{le nombre de parties génératrices de $G$ est divisible par $\abs{G}$}. De plus, on a le résultat arithmétique suivant (en spécialisant ce qui précède au groupe $\mbb Z/n\mbb Z$) :
     $$\boxed{\sum_{d\mid n}\mu(d)t^{n/d}\equiv 0 \ [n]}$$
\end{corol}

\section{Génération du groupe symétrique}

\subsection{Isométries et interstices}
On considère un groupe abélien (non nécessairement fini) $G$ \textbf{avec au moins $3$ éléments} et son plongement de Cayley $G\longrightarrow \sym_G$. Dans la suite, on confondra $G$ et son image par le plongement de Cayley, de sorte que l'on écrira $G\subseteq \sym_G$. On se pose la question suivante (fréquente en théorie de Galois par exemple) : que faut-il ajouter à $G$ pour engendrer $\sym_G$ ? \\

On commence par un résultat général :

\begin{lem}
    \normalfont Soit $X$ un ensemble. L'ensemble des permutations de $X$ à support fini, $\sym_X^f$, est le sous-groupe de $\sym_X$ engendré par les transpositions.
\end{lem}

\begin{proof}
    Clairement $\sym_X^f$ est un sous-groupe de $\sym_X$ qui contient les transpositions. Ensuite, si $\sigma \in \sym_X^f$, considérons $S$ son support fini et $\eta \in \sym_S$ la restriction naturelle de $\sigma$. Puisque $S$ est fini, $\sym_S$ est engendré par les transpositions, ce qui permet d'écrire $\eta$ puis $\sigma$ comme un produit de transpositions.
\end{proof}

\begin{defin}
    \normalfont (Isométries modulo $H$) Soit $H$ un sous-groupe de $G$. Un élément $\sigma \in \sym_G$ est une \textbf{isométrie modulo $H$} si pour tous $x,y\in G$ on a :
    $$\sigma(x)-\sigma(y) \equiv x - y \ [H]$$
    Les isométries modulo $H$ forment un sous-groupe de $\sym_G$ contenant $G$, noté $O(H)$.
\end{defin}

\begin{defin}
    \normalfont (Interstices de $K$) Soit maintenant $K$ un sous-groupe de $\sym_G$ contenant $G$ (on dira qu'un tel groupe est de Cayley). Un élément $\delta\in G$ est un \textbf{interstice} de $K$ si l'une des condtions suivantes (équivalentes) est vérifiée :
    \begin{itemize}
        \item $(0\ \delta)\in K$
        \item $\exists g\in G\ (g\ g+\delta)\in K$
        \item $\forall g\in G\ (g\ g+\delta)\in K$
    \end{itemize}
    L'ensemble des interstices de $K$ forme un sous-groupe de $G$ noté $\Delta(K)$.
\end{defin}

\begin{proof}
    \textbf{Les trois conditions sont équivalentes :} la troisième entraîne clairement la première, la première entraîne la seconde, et si la seconde est vraie, prenons un $g$ tel que $(g\ g+\delta)\in K$ ; soit $h\in G$, on a :
    $$(h\ h+\delta) = (h-g) \circ (g\ g+\delta)\circ (g-h)\in K$$
    car $K$ est de Cayley. \\

    \textbf{$\Delta(K)$ est un sous-groupe de $G$ :} on a clairement $0\in \Delta(K)$. Soient $x,y\in \Delta(K)$, on a $(0\ x)\in K$ donc $(x\ 0)\in K$ donc $-x\in \Delta(K)$. Ensuite $(x\ x+y)\in K$ donc en conjuguant, dans le cas où $x$ et $x+y$ sont non-nuls :
    $$(0\ x+y)=(x\ x+y)(0\ x)(x\ x+y)\in K$$
    donc $x+y\in \Delta(K)$. Le cas contraire est immédiat.
\end{proof}

On a ainsi défini une application croissante $O:\mathcal S(G)\longrightarrow \mathcal S_c(\sym_G)$ (sous-groupes de Cayley) et une application croissante $\Delta : \mathcal S_c(\sym_G)\longrightarrow \mathcal S(G)$. Elles ne sont pas réciproques l'une de l'autre en général, mais on a tout de même :

\begin{theorem}
    \normalfont $\Delta$ est surjective, $O$ est injective, et :
    $$\boxed{\Delta \circ O=\op{id}}$$
    De plus, $O$ possède un adjoint à gauche, $L$ donné par :
    $$L(K)=\left\langle\sigma(x)-\sigma(y)-x+y\mid \sigma \in K, x,y\in G\right\rangle$$
\end{theorem}

\begin{proof}
    Il suffit de montrer $\Delta \circ O=\op{id}$. Soit $H$ un sous-groupe de $G$, on a $H\subseteq \Delta\circ O(H)$ car étant donné un $h\in H$, $(0\ h)$ est bien une isométrie modulo. Ensuite, si $\delta \in \Delta \circ O(H)$, alors $(0\ \delta)\in O(H)$. Puisque $G$ a au moins $3$ éléments, il existe $x\in G\setminus\{0,\delta\}$ de sorte que :
    $$(0\ \delta)(x) - (0\ \delta)(0)\equiv x-0 \ [H]$$
    Autrement dit $x-\delta-x\in H$ donc $\delta\in H$. L'adjonction entre $L$ et $O$ est claire.
\end{proof}

\begin{defin}
    \normalfont (Sous-groupe de Cayley engendré par une partie) \\
    Si $S\subseteq \sym_G$, on note $((S))$ le plus petit sous-groupe de Cayley contenant $S$. Clairement, $((S))=\langle G \cup S\rangle $.
\end{defin}

\begin{prop}
\normalfont
    On a la relation suivante, pour $\tau = (x\ y)$ :
    $$\boxed{\Delta((\tau))=\langle y-x \rangle}$$
    Plus généralement, pour $\tau_i=(x_i\ y_i)$, avec $1\leq i\leq n$, on a :
    $$\boxed{\Delta((\tau_1,\dots, \tau_n))=\langle \delta_1,\dots,\delta_n\rangle}$$
    avec $\delta_i=y_i-x_i$.
\end{prop}

\begin{proof}
    Clairement $\Delta((\tau_1,\dots,\tau_n))\supseteq \langle \delta_1\dots,\delta_n \rangle$ car $\tau_i \in ((\tau_1,\dots,\tau_n))$. Ensuite la magie opère : $\tau_i \in O(\langle \delta_1,\dots,\delta_n\rangle )$, donc $((\tau_1,\dots,\tau_n))\subseteq O(\langle \delta_1\dots \delta_n \rangle)$ (car c'est un sous-groupe de Cayley), et par croissance de $\Delta$ :
    $$\Delta((\tau_1,\dots,\tau_n))\subseteq  \Delta \circ O(\langle \delta_1,\dots,\delta_n\rangle)=\langle \delta_1,\dots,\delta_n\rangle$$
    d'où la conclusion.
\end{proof}

La proposition suivante motive complétement cette section : on ramène la question de générer le groupe symétrique (à support fini) à la question plus simple de générer $G$.

\begin{prop}
    \normalfont Soit $K$ un sous-groupe de Cayley. On a :
    $$\boxed{K\supseteq \sym_G^f \iff \Delta(K)=G}$$
    Si $G$ est fini, cela donne :
    $$\boxed{K=\sym_G\iff \Delta(K)=G}$$
\end{prop}

\begin{proof}
    Les équivalences suivantes sont claires (par le lemme vu précedemment) :
    $$K\supseteq \sym_G^f \iff K \ \text{contient toutes les transpositions} \iff \Delta(K)=G$$
\end{proof}

Le fait important qui découle de toutes ces généralités est le suivant :

\begin{theorem}
    \normalfont En gardant les notations précédentes, on a :
    $$\boxed{((\tau_1,\dots,\tau_n))\supseteq \sym_G^f \iff \langle \delta_1,\dots, \delta_n\rangle = G}$$
    Et quand $G$ est fini :
    $$\boxed{((\tau_1,\dots,\tau_n)) = \sym_G \iff \langle \delta_1,\dots, \delta_n\rangle = G}$$
\end{theorem}

\subsection{Applications}
Voyons une application directe :

\begin{theorem} 
\normalfont (Génération de $\sym_n$ avec un $n$-cycle et une transposition) \\
    Soit $n\geq 3$ et $\tau=(i\ j)\in \sym_n$. Le cycle $(1\ 2\ \dots\ n)$ et la transposition $\tau$ engendrent $\sym_n$ si et seulement si $n\wedge (j-i)=1$. En particulier, si $p\geq 3$ est premier, un $p$-cycle et une transposition engendrent toujours $\sym_p$. 
\end{theorem}

\begin{proof}
    Appliquer ce qui précède à $G=\mbb Z_n$.
\end{proof}

Une autre application est que $x\mapsto x+1$, la transposition $(0\ 2)$ et la transposition $(0\ 3)$ engendrent toutes les permutations à support fini de $\mbb Z$.

\begin{remark}
    \normalfont Tout ceci ne fonctionne pas pour un groupe d'ordre $2$, puisque $\sym_2$ est égal à $((\op{id}))$, alors que, en notant $\op{id}=(1\ 1)$, on n'a pas $1-1$ premier avec $2$.
\end{remark}

\subsection{Description des isométries}
On reprend les notations de la partie sur le groupe symétrique. Soit $H$ un sous-groupe de $G$, et $f\in O(H)$. On a, par définition :
$$f(x)-f(y)\equiv x-y \ [H]$$
pour tous $x,y\in G$. On peut réécrire cela ainsi :
$$f(x)-x\equiv f(y)-y\ [H]$$
Autrement, dit, $f(x)-x$ est une constante modulo $H$, on la note $c(f)\in G/H$. De plus, une permutation est une isométrie modulo $H$ si et seulement si il existe une telle constante modulo $H$.

\begin{prop}
    \normalfont Cela définit un morphisme de groupes :
    $$O(H) \arw{c} G/H$$
\end{prop}

\begin{proof}
    Il s'agit de montrer que pour $f,g\in O(H)$, on a $c(f\circ g)=c(f)+c(g)$. En effet, on observe pour $x\in G$ :
    $$f\circ g(x)-x\equiv f\circ g(x) - g(x) + g(x) -x \equiv c(f) + c(g) \ [H]$$
    en notant abusivement $c(f)$ un représentant de $c(f)$.
\end{proof}

\begin{prop}
    \normalfont Le noyau de $c$ est constitué des morphismes stabilisant les classes modulo $H$. On a alors la suite exacte suivante (quitte à ordonner les classes modulo $H$) :
    $$1\longrightarrow (\sym_H)^{[G:H]}\longrightarrow O(H) \longrightarrow G/H\longrightarrow 1$$
\end{prop}

\begin{proof}
    La flèche $(\sym_H)^{[G:H]}\longrightarrow O(H)$ correspond à l'action sur chaque classe de $(\sym_H)^{[G:H]}$ sur $G$, qui est fidèle et se fait bien par isométries modulo $H$ puisqu'il existe une constante modulo $H$ ($0$) pour chaque élément dans l'image de ce morphisme. Ensuite, le morphisme $c$ est surjectif puisque le diagramme suivant commute :
\[\begin{tikzcd}
	G && {O(H)} \\
	& {G/H}
	\arrow[hook, from=1-1, to=1-3]
	\arrow[two heads, from=1-1, to=2-2]
	\arrow["c", from=1-3, to=2-2]
\end{tikzcd}\]
\end{proof}

\begin{corol}
        \normalfont On a donc directement, lorsque $G$ est fini, en notant $g$ le cardinal de $G$ et $h$ le cardinal de $H$ :
    $$\boxed{\abs{O(H)}=\frac{(h!)^{\frac g h}g}h}$$
\end{corol}

\section{Perspectives}

Une question naturelle est de savoir si l'on peut généraliser la convolution aux groupes non commutatifs, en sommant seulement sur les sous-groupes distingués. Malheureusement, on y perd la commutativité et l'associativité (le problème pour l'associativité étant qu'on peut avoir $H$ distingué dans $K$ et $K$ distingué dans $G$ sans que $H$ ne soit distingué dans $G$). Une piste de généralisation est peut-être d'appliquer cela aux modules sur un anneau principal dont les quotients sont de cardinal fini (par exemple $k[X]$ avec $k$ un corps fini). \\\\

De même, les considérations sur le groupe symétrique ne fonctionnent plus lorsque le groupe de départ est non commutatif (la définition d'isométrie doit être changée pour cela).  \\\\

En discutant avec mon collègue Rafik SOUANEF, on s'est rendus compte que la fonction $\mu$ semblait être liée au nombre de sous-groupes à structure fixée - c'est à dire, étant donné un groupe abélien fini $A$, le nombre de sous-groupes de $A$ isomorphes à un groupe $B$ fixé. Il se trouve que l'on a réussi à donner une formule pour ce nombre, en fonction des facteurs invariants de $A$ et $B$, cf \cite{sousgroupes}. On a même trouvé une seconde démonstration qui n'utilise pas la convolution. Encore une fois, cette formule avait déjà été trouvée dans \cite{delsarte} par une méthode différente, ce dont on s'est aperçus plus tard.

\nocite{*}
\bibliographystyle{plainurl}
\bibliography{refs}

\end{document}